\documentclass[12p]{amsart}
\usepackage[top=1.5in, bottom=1.5in, left=1.5in, right=1.5in]{geometry}
\usepackage{fourier} 
\usepackage{latexsym}
\usepackage{amssymb}
\usepackage{amsmath}
\usepackage{amscd}
\usepackage{graphicx}
\usepackage{float}
\usepackage{mathrsfs} 
\usepackage{enumerate}
\usepackage{tikz}

\tikzset{node distance=2cm, auto}

\newtheorem{theorem}{Theorem}[section]
\newtheorem{proposition}[theorem]{Proposition}
\newtheorem{lemma}[theorem]{Lemma}
\newtheorem{corollary}[theorem]{Corollary}

\theoremstyle{definition}
\newtheorem{definition}[theorem]{Definition}

\theoremstyle{remark}
\newtheorem{example}[theorem]{Example}
\newtheorem{remark}[theorem]{Remark}

\def\<{\langle}
\def\>{\rangle}

\def\RR{\mathbb{R}}

\def\CC{\mathbb{C}}
\def\ZZ{\mathbb{Zie95}}
\def\PP{\mathbb{P}}

\def\H{\mathcal{H}}

\def\H{\mathcal{H}}

\def\N{\mathcal{N}}

\def\cX{\check{X}}

\def\Conv{\operatorname{Conv}}

\def\GL{\operatorname{GL}}
\def\Hom{\operatorname{Hom}}

\def\sgn{\operatorname{sgn}}
\def\d{\operatorname{d}}

\title{Defective dual varieties for real spectra}
\date{\today}
\author{Jens Forsg{\aa}rd}
\address{Universit{\'e} de Gen{\`e}ve, Math{\'e}matiques, Villa Battelle, 1227 Carouge, Suisse.}
\email{Jens.Forsgaard@unige.ch}

\begin{document}

\begin{abstract}
We introduce an invariant of a finite point configuration $A \subset \RR^{1+n}$
which we denote the \emph{cuspidal form} of $A$. We use this invariant to extend Esterov's
characterization of dual defective point configurations to exponential sums;
the dual variety associated to $A$ has codimension at least $2$
if and only if $A$ does not contain any iterated circuit.
\end{abstract}

\maketitle

\section{Introduction}
\label{sec:Intro}

The main undertaking of \emph{fewnomial theory} is to bound the number of connected components
of the positive part of a variety defined by a system of equations solely in terms of the number
of variables $n$ and the total number of monomials $N$ appearing in the system.
Since the constitutive monograph \cite{Kho91} of Khovanski{\u\i}, fewnomial theory has often
been studied alongside the theory of exponential sums. After all, the 
coordinatewise exponential map $\exp \colon \RR^n \rightarrow \RR_+^n$
is a diffeomorpism and replacing monomials $z^\alpha$ by
exponentials $e^{\<w, \alpha\>}$ the fundamental examples (read: Descartes' rule of signs)
remain valid.

Lately, fewnomial theory has also been studied from the viewpoint of Gel'fand, Kapranov, and Zelevinsky's 
``$A$-philosophy.'' In this approach, one considers the family of all polynomials
which can be expressed using a fixed set (the \emph{support set}) $A$ of exponent vectors $\alpha$.
For example, in \cite{BD16} the bound on the number of positive solutions of a system of
equations supported on a circuit was sharpened by considering in addition the
combinatorics of $A$.

Seldom has the two approaches been combined. We are aware only of \cite{RR16},
where the \emph{$A$-discriminant} and its \emph{Horn--Kapranov uniformization} was generalized
to the case of exponential sums. We put ourself in this setting, and consider the family
\[
(\CC^\times)^A = \bigg\{ \, f(w) = \sum_{\alpha \in A} c_\alpha\, e^{\<w,\alpha\>} \, \bigg| \, c_\alpha \in \CC^\times \,\bigg\},
\]
where $A\subset \RR^{1+n}$ is a finite set, and $N = \# A$. We have many names for those we love;
the support set $A$ is also known as the \emph{Bohr spectrum} (or, simply, \emph{spectrum}) of an exponential sum
$f \in (\CC^\times)^A$.

The work presented in this article emerged from an innocent
question of whether a theorem of Katz \cite{Kat73a} holds also in the framework of 
exponential sums. The answer is affirmative, as we show in Theorem~\ref{thm:Hessian}. This theorem
has already seen an application in \cite{FNR17} to reduce a fewnomial hypersurface bound from exponential
to subexponential in the dimension.
Our main line of thought goes, however, in a slightly different direction. 

We associate to the spectrum $A\subset \RR^{1+n}$ a combinatorial invariant in the form of a
homogeneous polynomial $P_A(t)$ of degree $n$, which we call the \emph{cuspidal form}
of $A$, see Definition~\ref{def:CharacteristicPolynomial}. Here, $t$ denotes the parameters of the Horn--Kapranov uniformization of the $A$-discriminant $\cX_A$ and, hence, the cuspidal form depends also on a choice of \emph{Gale dual} of $A$. The name ``cuspidal form'' reflects the fact that $P_A(t)$ describes the preimage of the cuspidal
locus of $\cX_A$ under the Horn--Kapranov uniformaziation. 
In particular, as one observes immediately, the configuration
$A$ is \emph{dual defective} if and only if $P_A(t)$ is trivial.

The core part of this work is to describe the properties of the cuspidal form $P_A(t)$ as an invariant of the
spectrum $A$. Our main, technical, results are concerned with factorizations.
For example, if $A$ is a \emph{diagonal configuration} (Definition~\ref{def:Diagonal}),
then the cuspidal form $P_A(t)$ factors as a product of the cuspidal forms of the
\emph{diagonal configurations} of $A$.

We spend a fair amount of energy describing linear factors of $P_A(t)$.
These are of two distinct types. The first type consists of
linear factors corresponding to rows of the Gale dual. These factors correspond to
points $\alpha \in A$ such that $A \setminus \{\alpha\}$ is dual defective.
Such factors are studied in \S\ref{ssec:SubConfigurations} and \S\ref{sec:LinearFactors}.
The complementary type corresponds to discriminant varieties embedded into the cuspidal locus
of $\cX_A$. This generalizes results of \cite{MT15}, where the cuspidal form 
first appeared in the special case of $n=1$ (when $P_A(t)$ is itself a linear form).

As an application, we extend Esterov's characterization from \cite{Est10} and \cite{Est13} of dual defective point 
configurations to the case of exponential sums: 
A point configuration $A$ is dual defective if and only if it does not contain any 
\emph{iterated circuit} (Theorem~\ref{thm:DualDefectiveIteratedCircuit}).
Our proof uses in addition to the properties of the cuspidal form only the pidgeon hole principle.

Finally, in \S\ref{sec:Bivariate} we study the special case $n=2$ when $P_A(t)$
is a quadratic form. Defined only up to choice of coordinates of the parameter space $\CC^m$ of the Horn--Kapranov uniformization, it is natural to consider the rank of $P_A(t)$. 
We consider instead the signature of $P_A(t)$, which is a well-defined invariant of $A$
if we restrict to real Gale duals.
It turns out (Theorem~\ref{thm:QuadraticFormRank}) that the rank of $P_A(t)$
is at most $3$, independent on the number of variables $m$.
The degenerate case that the rank is at most $2$ occurs if and only if
the spectrum $A$ is contained in a conic section; the type of the conic section
is described by the signature of $P_A(t)$.
These results relate to the study of self-dual toric varieties as follows.
Given a two-dimensional (integral) point configuration $A$, the cusps of of the
associated toric variety $X_A$ are (intuitively) contained in two lines. Hence, if 
$X_A$ is self dual, then $P_A(t)$ factors into linear forms, 
implying by Theorem~\ref{thm:QuadraticFormRank} that $A$ is contained in a conic section.
For an explicit example, see the description of the Togliatti surface in \cite[\S4.1]{DP17}, 
whose spectrum consists of six points on an ellipse.

It is possible toconsider for example complex spectra instead of real.
The $A$-discriminant variety, in the sense of exponential sums, remains well-defined. 
The characterization
of dual defective point configurations, as those who do not contain an iterated
circuit, remains valid. However, for complex spectra, the relationship to combinatorics
is lost, as we can for example not talk about Newton polygons. We have chosen to stay in the real world
as this is the most interesting case for applications.

\subsection{Acknowledges}
I would like to thank Professor Ragni Piene for the helpful discussions,
and Professor J.~Maurice Rojas for the inspiration.

\section{Prelude: Gale Duality and the Horn--Kapranov Uniformization}
\label{sec:Gale}
\subsection{The setup}
By ordering the elements of $A$ we obtain an isomorphism $\CC^A \simeq \CC^N$.
With an ordering chosen, by abuse of notation, we identify $A$ with the $(1+n)\times N$-matrix $A = \big(\alpha_1, \dots, \alpha_N\big)$.
Doing so, we often write $c_k$ for the coefficient of $e^{\<w, \alpha_k\>}$, and identify an exponential sum
$f$ with its (column) coefficient vector $c = (c_1, \dots, c_N)^T$.

The properties of the exponential sum $f$ which we are interested in (e.g., the existence of a singular point)
are invariant of linear changes of variables. That is, if we denote by $M_{N,n}(\RR)$ the space of all such matrices $A$, then we are interested in the orbits under the left action of $\GL_{1+n}(\RR)$.

If $A \subset \ZZ^{1+n}$, then the exponential sum is said to be polynomial, and the configuration
$A$ is said to be algebraic. In this case, we can associate to $f$ the polynomial 
$g(z) = \sum_{\alpha \in A} c_\alpha\, z^\alpha$. The analytic variety $Z(f)$ is an infinite covering
of the quasi-affine variety $Z(g) \subset (\CC^\times)^{1+n}$. We often descend to the algebraic case in
examples.

Two natural assumptions are imposed. Firstly, $A$ is assumed
to be pseudo-homogeneous. That is, we assume the existence of a linear form  
$\xi \in \Hom(\RR^{1+n}, \RR)$ such that $\<\xi, \alpha\> = 1$ for all $\alpha \in A$.
Secondly, it is assumed that $\xi$ is unique. These two assumptions are equivalent to
that the matrix $A$ is of full rank (equal to $1+n$) with the all ones vector in its row span.
In particular, the \emph{Newton polygon} $\N = \Conv\big(\alpha_1, \dots, \alpha_N\big)$
has dimension $n$. We do not assume that the columns of $A$ are distinct.

We associate to $A$ the map $\exp_A$ given by
\[
\exp_A\colon \CC^{1+n} \rightarrow (\CC^\times)^N, \quad z\mapsto \big(e^{\<w, \alpha_1\>}, \dots, e^{\<w, \alpha_N\>}\big)^T,
\]
where $\CC^N$ should be considered as the dual space of $\CC^A$. 
In the algebraic case, the map $\exp_A$ parametrizes a toric variety denoted $X_A \subset (\CC^\times)^N$. 
Its projectively dual $\cX_A \subset (\CC^\times)^A$ is known as the \emph{$A$-discriminantal variety}, see \cite{GKZ94}.
In the case that $A$ is algebraic and $\cX_A$ is a hypersurface, 
its defining polynomial $D_A$ is called the \emph{$A$-discriminant}.
In the algebraic case, by definition, $D_A(c) = 0$ if and only if the polynomial $g$ 
has a singular point in $(\CC^\times)^n$.
In the general case we define $\cX_A \subset (\CC^\times)^A$ by that $f \in \cX_A$
if and only if the exponential sum $f$ has a singular point in $\CC^n$. In general, $\cX_A \subset (\CC^\times)^A$ is not algebraic for real spectra.

\begin{example}
A point configuration $A$ is said to be a \emph{pyramid}
if all points but one is contained in some strict affine subspace.
It is straightforward to check that if $A$ is a pyramid, then the discriminant locus 
$\cX_A$ is empty.
\end{example}

\begin{example}
The \emph{codimension} $m$ of a spectrum $A$ is given by $m = N - n - 1$.
A point configuration of codimension $1$, which is not a pyramid,
is said to be a \emph{circuit}. Let us consider a circuit in two dimensions.
Using the $\GL_{1+n}(\RR)$-action, we can assume that $A$ is of the form
\[
A = \left(\begin{array}{cccc}
1 & 1 & 1 & 1 \\
0 & 1 & 0 & \alpha_{11} \\
0 & 0 & 1 & \alpha_{12}
\end{array}\right).
\]
The formula of the discriminant $D_A(c)$ in the algebraic case, from \cite[p.~274]{GKZ94},
generalizes to the binomial exponential expression
\[
D_A(c) = (\alpha_{11}+\alpha_{12} - 1)^{\alpha_{11}+\alpha_{12} - 1} c_0 c_1^{\alpha_{11}} c_2^{\alpha_{12}}+ (-c_0 \alpha_{11})^{\alpha_{11}} (-c_0 \alpha_{12})^{\alpha_{12}}.
\]
The multivaluedness of the exponential functions require some caution when handling this
expression.
For this reason, we refrain from using the $A$-discriminant $D_A(c)$ in our analysis.
\end{example}

\begin{remark}
In examples, it is more convenient to consider a family of inhomogeneous $n$-variate exponential sums.
For an exponential sum $f$ in the variables $z_1, \dots, z_n$, the corresponding pseudo-homo\-geneous
exponential sum is given by $e^{z_0} f$. For a family of $n$-variate inhomogeneous polynomials
with support set $A$ this corresponds to adjoining a top row of all ones in its matrix representation.
In this case, $\xi = (1, 0, \dots, 0)$.
\end{remark}

\subsection{The Horn--Kapranov uniformization}
Our main tool is the Horn--Kapranov uniformization 
of the $A$-discriminant hypersurface from \cite{Kap91}.
For exponential sums, this map was deduced in \cite[Thm.~1.7]{RR16}.
Though, step-by-step, the standard deduction of the Horn--Kapranov
uniformization (\emph{cf.} \cite[Prop.~4.1]{DFS07}) is sufficient to cover also
the case of exponential sums.

A \emph{Gale dual} of $A$ is a matrix $B$ yielding an exact sequence
of $\CC$-vector spaces
\begin{center}
\begin{tikzpicture}
  \node (O0) {0};
  \node (O1) [right of=O0] {$\CC^{m}$};
  \node (O3) [right of=O1] {$\CC^{A}$};
  \node (O5) [right of=O3] {$\CC^{1+n}$};
  \node (O6) [right of=O5] {0.};
  \draw[->] (O0) to node {} (O1);
  \draw[->] (O1) to node {$B$} (O3);
  \draw[->] (O3) to node {$A$} (O5);
  \draw[->] (O5) to node {} (O6);
\end{tikzpicture}
\end{center}
Since $A$ has full rank, the rank of $B$ is $m$. 
Let us denote the rows of $B$ by $\beta_k$ for $k= 1, \dots, N$,
and let us introduce coordinates $t$ in $\CC^m$. 
Let $\CC_B^m$ denote the inverse image of $(\CC^\times)^A$ under the map $B$.
The \emph{exceptional locus} $\H_B = \CC^m \setminus \CC_B^m$
is the union of the linear subspaces $\beta_k^*$.
Assuming that $A$ is not a pyramid (i.e., assuming that
no $\beta_k = 0$), then $\H_B$ is a
central hyperplane arrangement.

\begin{theorem}[Kapranov, Rojas--Rusak]
\label{thm:HornKapranov}
The dual variety\/ $\cX_A$ is parametrized by the map
\[
\Phi \colon  (\CC^\times)^{1+n}\times \CC^m_B \rightarrow (\CC^\times)^A, \quad \Phi(\omega; t) = \exp_A(\omega)*(Bt),
\]
where $*$ denotes component-wise multiplication. \qed
\end{theorem}

The map $\Phi$ is far from injective; the parameter space and $(\CC^\times)^A$
has the same dimensions, $N = 1+ n + m$. However, that $\Phi$ is homogeneous in $t$ and 
the existence of the linear form $\xi$ implies that $\Phi$ para\-metrizes a strict (multivalued)
analytic subvariety of $(\CC^\times)^A$. 

\begin{remark}
\label{rem:ReducedDiscriminant}
In the algebraic case, the $A$-discriminantal polynomial has $1+n$ homogeneities, arising from the matrix
$A$. In Kapranov's paper \cite{Kap91}, these homogeneities was removed by composition
with the map $\exp_B\colon (\CC^\times)^A \rightarrow (\CC^\times)^m$. The composite map is, in this algebraic case,
a rational function of $t$. To avoid real powers of linear forms we
settle for the map $\Phi$ of Theorem~\ref{thm:HornKapranov}.
\end{remark}

\begin{remark}
The exponential sum $f = \Phi(\omega; t)$ has a singular point with coordinates $w = -\omega$.
\end{remark}

\subsection{Gale duality}
Let us describe a well-known property of Gale duality. 
Since $A$ has full rank, 
we can find a non-vanishing maximal minor $|A_1|$ of $A$.
After possibly rearranging the columns of $A$, we can write it in the block form $A = (A_1, A_2)$.
Let us extend $A$ to a square matrix $\bar{A}$, with determinant $= 1$, and inverse $\bar{B}$:
\begin{equation}
\label{eqn:UnimodularExtensions}
\bar{A} = \left(\begin{array}{cc}
 A_1 & A_2\\ * & *
 \end{array}\right)
\quad \text{and} \quad
\bar{B}
= \left(\begin{array}{cc}
 * & B_2 \\ * & B_1
\end{array}\right).
\end{equation}
From $\bar{B}$ we obtain the Gale dual 
\begin{equation}
\label{eqn:B}
B 
= \left(\begin{array}{c}
  B_2 \\  B_1
\end{array}\right).
\end{equation}

\begin{lemma}
\label{lem:Schur}
Let $B$ be any Gale dual of $A$, and let $\sigma \subset [N]$ be a choice of\/ $1+n$ indices
with sign $\sgn(\sigma) = \prod_{k\in \sigma}(-1)^k$.
Denote by $A_\sigma$ the maximal cofactor of $A$ obtained by keeping the
columns indexed by $\sigma$ and let $B_\sigma$ denote the complementary maximal cofactor of\/ $B$
obtained by deleting the rows indexed by $\sigma$. Then,
there is a nonzero constant $C(B)$, independent of $\sigma$, such that 
\[
|A_\sigma| = C(B)\sgn(\sigma)\, |B_\sigma|.
\]
\end{lemma}

\begin{proof}
Two Gale duals differ only by multiplication by a a matrix $T \in \GL_m(\RR)$. Thus, it suffices to prove
the theorem for one explicit Gale dual $B$. Indeed, if $B' = BT$, then $C(B') = C(B)\, |T|$.

Let $A_1, A_2, B_1, B_2, \bar{A}$, and $\bar{B}$ be as in \eqref{eqn:UnimodularExtensions}. 
We claim that $C(B) = 1$ in this case.
Since $|A_1| \neq 0$, we can perform
row operations on $\bar{A}$, which does not alter its determinant, to eliminate all entries in the bottom left block.
This corresponds to multiplication by some lower triangular matrix $E$ of the form
\begin{equation}
\label{eqn:SchurElementaryMatrix}
E = \left(\begin{array}{cc}
 I & 0\\ * & I  
 \end{array}\right).
\end{equation}
Notice that $E^{-1}$ has the same form as $E$. We have that $E \bar{A} \bar{B} E^{-1} = I$. 
It follows that 
\[
\bar{B} E^{-1} = \left(\begin{array}{cc}
 * & B_2 \\ * & B_1  
 \end{array}\right),
\quad \text{and} \quad 
E \bar{A} = \left(\begin{array}{cc}
 A_1 & A_2\\ 0 & B_1^{-1}
 \end{array}\right).
\]
Since $|E \bar{A}| = 1$ we conclude that $|A_1| = |B_1|$. 

Let $\sigma$ be some other choice of indices. If $|A_\sigma| = |B_\sigma| = 0$ then there is nothing to prove. Also, since $A^T$ is a Gale dual of $B^T$, it suffices to consider the case $|A_\sigma| \neq 0$.

Let $P_\sigma$ be a permutation matrix such that $AP_\sigma = (A_\sigma, A_{\kappa})$, where the indices of $\sigma$
and $\kappa = [N] \setminus \sigma$ are ordered increasing. As above, we find an elementary matrix $E$ of the form \eqref{eqn:SchurElementaryMatrix} such that
\[
E \bar{A} P= \left(\begin{array}{cc}
 A_\sigma & A_\kappa \\ 0 & B_\sigma^{-1}
 \end{array}\right).
 \]
It follows that $|A_\sigma| = |P_\sigma|\, |B_\sigma|$, which concludes the proof.
\end{proof}

\begin{remark}
In the algebraic case, it is natural to assume that $\ZZ A = \ZZ^{1+n}$,
in which case $\bar{A}$ can be chosen as an integer unimodular matrix.
That is, also $\bar{B}$ is integer unimodular, and hence $B$ is a Gale dual in the combinatorial sense.
In this case we have that $C(B) = \pm 1$, depending on $B$.
\end{remark}

\section{The Cuspidal Form}
\label{sec:CuspidalForm}
In this section we define the cuspidal form $P_A(t)$. To simplify notation,
we impose the assumption that the top row of $A$
is equal to the all ones vector, and we let $\hat A$ denote the $n\times N$
matrix obtained from $A$ by deleting the top row.

Let $B$ denote a Gale dual of $A$. The polynomial $P_A(t)$
depends on $B$ only up to an affine change of coordinates in $\CC^m$;
we consider this dependence to be implicitly understood from the fact that $P_A(t)$
is written as a polynomial in the variables $t$.

In slight deviation from the notation in Lemma~\ref{lem:Schur},
we let $\Sigma$ denote the set of all subsets $\sigma \subset [N]$
of cardinality $n$. For each $\sigma \in \Sigma$,
 let $\hat A_\sigma$ denote the maximal cofactor of $\hat A$ obtained by deleting all
 columns corresponding to indices $k \notin \sigma$.

\begin{definition}
\label{def:CharacteristicPolynomial}
We define the \emph{cuspidal form} $P_A(t)$ to be
\begin{equation}
\label{eqn:CharacteristicPolynomial}
P_A(t) = \sum_{\sigma\in \Sigma} \left|\hat A_\sigma \right|^2 \, \prod_{k \in \sigma}  \<\beta_k, t\>.
\end{equation}
\end{definition}

The cuspidal form $P_A(t)$ is a homogeneous form
of degree $n$ in the coordinates $t$.
Hence it defines, in the case that it is nontrivial,
a hypersurface in $\PP^{m-1}$. (Except, of course, for 
the case $m=1$.)
As $P_A(t)$ is defined only up to choice of coordinates
one can not, in general, ask for a combinatorial interpretation
of its coefficients. When there is a canonical choice of coordinates
there is, however, reasonable interpretation.

\begin{example}
\label{ex:PyramidHasTrivialCharacteristicPolynomial}
Let $A$ be a pyramid. That is, there is a point $\alpha \in A$ such that
$A\setminus \{\alpha\}$ is contained in an affine space of dimension $n-1$.
This is equivalent to that $\beta_\alpha = 0$. 
Thus, if $\alpha \in \sigma$ then $\<\beta_\alpha, t\> = 0$, and if $\alpha \neq \sigma$
then $|\hat A_\sigma| = 0$. Hence, each term of \eqref{eqn:CharacteristicPolynomial}
vanishes, implying that $P_A(t)$ is trivial.
\end{example}

\begin{example}
\label{ex:PinCodimensionOne}
Let $m = 1$, and let $A_k$ denote the $(n+1)\times (n+1)$ submatrix
of $A$ obtained by deleting the $k$th entry and assume, without loss of generality,
that $\alpha_1 = (1, 0, \dots, 0)^T$. Then, each $\sigma \in \Sigma$ containing $1$ 
has $|\hat A_\sigma| = 0$,
and each $\sigma$ not containing $1$ has $|\hat A_\sigma| = |A_k|$ where $k = k(\sigma)$ is the unique index greater than $1$ not contained in $\sigma$.
By Lemma~\ref{lem:Schur} we can choose a Gale dual $B$ according to the rule that 
$\beta_k =  (-1)^k |A_k|$.
Hence,
\[
P_A(t) 
= t^n \, \Big( \prod_{k > 1} (-1)^k |A_k|\Big)\sum_{k > 1}(-1)^k |A_k|  =   \sgn(n) \, t^n\, \prod_{k = 0}^{n+2} |A_k|
\]
where $\sgn(n) = (-1)^{{n \choose 2}}$.
Notice, in particular, that $P_A(t)$ is trivial if and only if $|A_k| = 0$ for some $k$,
which is equivalent to that $A$ is a pyramid.
\end{example} 

\subsection{The Cuspidal Locus of $\cX$}
\label{ssec:Cusps}
Recall the assumption that the top row of $A$ consist of the all ones vector. 
Let us dehomogenize $\Phi$ by setting $\omega_0 = 1$. That is, we consider
the parametrization map $\Phi\colon (\CC^\times)^n \times \CC^m_B \rightarrow \cX_A$.
In standard terminology, a point $f  = \Phi(\omega; t) \in \cX_A$ is a \emph{cusp} 
of $\cX_A$ if and only if the image of the pushforward
\[
\d\Phi \colon T_{(w; t)} \left( (\CC^\times)^n \times \CC^m_B\right) \rightarrow T_{f} \cX_A
\]
is an affine space of dimension at most $N-2 = n+m-1$. That is, $f$ is a cusp if and only if
the Jacobian matrix
\[
J_\Phi(\omega; t) = \left( \begin{array}{ccc}
\alpha_{11}\,\<\beta_1, t\>\,e^{\<\omega, \alpha_1\>}\, & \dots & \alpha_{N1}\,\<\beta_N, t\>\,e^{\<\omega, \alpha_N\>} \\
\vdots & \ddots & \vdots \\
\alpha_{1n}\,\<\beta_1, t\>\,e^{\<\omega, \alpha_1\>}\ & \dots & \alpha_{Nn}\,\<\beta_N, t\>\,e^{\<\omega, \alpha_N\>}\\
\beta_{11} e^{\<\omega, \alpha_1\>}\ & \dots & \beta_{N1} e^{\<\omega, \alpha_N\>} \\
\vdots & \ddots & \vdots \\
 \beta_{1m} e^{\<\omega, \alpha_1\>}\ & \dots & \beta_{Nm} e^{\<\omega, \alpha_N\>}
\end{array}\right).
\]
has rank $n+m-1$. (We drop the subindex $\Phi$ in the notation.)

\begin{theorem}
\label{thm:CuspidalLocus}
We have that $f = \Phi(\omega; t)$ is a cusp of\/ $\cX_A$ if and only if\/ $P_A(t) = 0$.
\end{theorem}

\begin{proof}
We need to determine when the maximal minors of $J(\omega; t)$ vanishes.
This does not depend on $\omega$: the common 
factors $e^{\<\omega, \alpha_k\>}$ of each column of $J(\omega; t)$ 
can be factored outside of any minor. We compute the maximal minors of $J(1; t)$.

For each submatrix $S \subset A$ obtained by deleting any number of columns, 
let $\hat S$ denote the matrix $S$ with the top
row (consisting of all ones) deleted. In particular, $\hat \alpha$ denotes
the column of $\hat A$ corresponding to the column $\alpha \in A$

Let  $J_{k}$ denote the maximal minor of $J(1; t)$
obtained by deleting the $k$th column.
Without loss of generality, we impose two assumptions.
Firstly, as permuting the columns only alters the value of a minor by $\pm 1$,
we can assume that $k = 1$.
Secondly, we assume that $\alpha_{11} = \dots = \alpha_{1n} = 0$
(i.e., that $\alpha_1 = (1, 0 \dots, 0)^T$). Indeed, multiplying $J(1; t)$ from the left by
\[
T = \left( \begin{array}{ccc}
I_n & - \hat \alpha_1 t^T \\
0 & I_m
\end{array}\right)
\]
corresponds to the translation $\hat \alpha \mapsto \hat \alpha - \hat \alpha_1$.
As $\left|T\right| = 1$, all maximal minors are invariant under this action.

Let $\Sigma'$ denote the set of all subsets $\sigma \subset [N]\setminus \{1\}$ of cardinality $n$,
and for each $\sigma\in \Sigma'$ set $\bar \sigma = \{1\}\cup \sigma$.
By the Laplace expansion in complementary minors (of sizes $n\times n$ and $m\times m$)
\[
J_{1} = \sum_{\sigma \in \Sigma'} \sgn(\sigma) \left( \prod_{k\in \sigma} \<\beta_k, t\>\right) \left|\hat A_\sigma\right| \, \left|B_{\bar \sigma}\right|.
\]
We make three observations.
Firstly, since $\alpha_1 = (1,0, \dots, 0)^T$ we have that
$\left|A_{\bar \sigma}\right| = \left| \hat A_{\sigma}\right|$. Secondly, 
according to Lemma~\ref{lem:Schur}  (assuming that $C(B) = 1$)
we have that $\left| B_{\bar \sigma}\right| = \sgn(\bar \sigma)\left| A_{\bar \sigma}\right|$.
Thirdly, we have that $\sgn(\sigma) = \sgn(\bar \sigma)$.
All in all, we conclude that
\[
J_{1} = \sum_{\sigma \in \Sigma'}\left|\hat A_\sigma\right|^2 \Big( \prod_{k\in \sigma} \<\beta_k, t\>\Big).
\]
Since $\hat \alpha_1 = 0$, we have that $|\hat A_\sigma| = 0$ for any $\sigma$ with $1 \in \sigma$.
Hence, $J_1 = P_A(t)$.
\end{proof}

\subsection{Katz' Theorem}
\label{sec:Katz}
In the algebraic case, it follows from Katz' dimension formula \cite{Kat73a}, see also \cite[\S 1.5]{GKZ94},
that the Jacobian $J(\omega; t)$ has full rank at $(\omega; t)$ if and only if
the Hessian matrix $H_f(w)$ of $f = \Phi(\omega; t)$ at the point $w = -\omega$ is singular.
This statement remains true if one generalizes to exponential sums.
It is possible to obtain a proof if this fact by making the appropriate modifications to
the exposition in \cite[\S 1.5]{GKZ94}; since we have introduced the cuspidal form \eqref{eqn:CharacteristicPolynomial}
we prefer the following more direct approach.

\begin{theorem}
\label{thm:Hessian}
Let $f = \Phi(\omega; t) \in \cX_A$. Then, the Hessian matrix $H_f(w)$ 
is singular at $w = -\omega$ if and only if\/ $P_A(t) = 0$. In particular,
if\/ $\cX_A$ is a hypersurface and $f$ is a smooth point on $\cX_A$,
then the Hessian matrix $H_f(-\omega)$ is nonsingular.
\end{theorem}

\begin{proof}
We only need to prove the first part.
As in the proof of Theorem~\ref{thm:CuspidalLocus} the torus action
can be quoted out. Hence, it suffices to consider the case $\omega = 0$, where the
singular point is located at $w = 0$. We have that
\[
f''_{kl}(0) = \sum_{j = 1}^N \alpha_{jk}\alpha_{jl}\< \beta_j, t\>, \quad 1 \leq k,l \leq n.
\]
In particular, the $p$th column $H_p$ of $H_f(1)$, for $p = 1, \dots, n$ is a sum of $N$ vectors
\[
H_p = \sum_{j = 1}^N\alpha_{jp}\< \beta_j, t\>\, \hat \alpha_j.
\]
For any two such columns, the $j$th summands are both multiples of $\hat \alpha_j$. 
In particular, we can expand the determinant
\[
\left|H_f(0)\right| = \sum_{\sigma \in \Sigma}\,  \sum_{\pi}\,
\left| \alpha_{k\pi(k)} \< \beta_{k}, t\>\, \hat \alpha_{k}\right|_{k\in \sigma}
\]
where $\pi$ runs over the set of all bijections $\pi \colon \sigma \rightarrow [n]$.
It follows that
\[
\left|H_f(0)\right| = \sum_{\sigma \in \Sigma}\, \Big( \prod_{k \in \sigma} \< \beta_{k}, t\> \Big)
\,  \sum_{\pi}\,\left| \alpha_{k\pi(k)}\, \hat \alpha_{k}\right|_{k\in \sigma}
=  \sum_{\sigma \in \Sigma}\, \Big( \prod_{k \in \sigma} \< \beta_{k}, t\> \Big)
\,  \left|\hat A_{\sigma}\right|^2,
\] 
which completes the proof.
\end{proof} 

\section{Properties of the Cuspidal Form}
\label{sec:CuspidalForm}

In this section we  investigate the cuspidal form $P_A(t)$
as a function of the spectrum $A$.

\subsection{Sub-configurations of $A$}
\label{ssec:SubConfigurations}

\begin{theorem}
\label{thm:RetsrictionOfPToHyperplane}
Let $\alpha \in A$ be such that $A \setminus \{\alpha\}$ has rank $1+n$,
let $t'$ be coordinates in $\beta_\alpha^*$ and $t = (t',t_m)$ coordinates
in $\CC^m$. Then,
\[
P_{A\setminus\{\alpha\}}(t') = P_A(t)\big|_{\beta_\alpha^*}.
\]
\end{theorem}

\begin{proof}
By assumption, if we write $A = (A', \alpha)$, then $A'$ is of rank $1+n$ 
and codimension $m-1$. We can chose a Gale dual $B$ of $A$ of the form
\[
B = \left(\begin{array}{cc} B' & * \\ 0 & 1 \end{array}\right),
\]
where $B'$ is a Gale dual of $A'$. The theorem follows from
that the restriction of $P_A(t)$ to $\beta_\alpha^*$ is given by $t_m = 0$.
First, for any $\sigma \in \Sigma$ with $N \in \sigma$,
we have that the term of $P_A(t)$ corresponding to $\sigma$
has the monomial $t_m$ as a factor.
Second, we have that $\<\beta_k, t\>$ restricted to $t_m = 0$
equals $\<\beta_k', t'\>$ for any $k = 1, \dots, N-1$.
\end{proof}

\begin{corollary}
\label{cor:ContainmentNontrivial}
If $A' \subset A$ is such that $P_{A'}$ is nontrivial, then $P_A$
is nontrivial.\hfill \qed
\end{corollary}

\subsection{Diagonal and upper diagonal configurations}

\begin{definition}
\label{def:Diagonal}
We say that a point configuration $A$ is \emph{upper diagonal} if it can, after acting by
$\GL_{1+n}(\RR)$, be written in the
form
\begin{equation}
\label{eqn:DefIteratedCircuit}
A = \left(\begin{array}{ccccc}
1 & 1 & 1 & \cdots & 1\\
0 & \tilde A_1 & * & \cdots & *\\
0 & 0 & \tilde A_2 & \cdots & *\\
\vdots & \vdots & \vdots & \ddots & \vdots\\
0 & 0 & 0 & \cdots & \tilde A_j
\end{array}\right),
\end{equation}
where we associate to each $\tilde A_j$ the point configuration 
$A_j$ defined by $\hat A_j = (0, \tilde A_j)$. 
The configurations $A_j$ for $j=1, \dots m$ are called the \emph{diagonal configurations} of $A$.
If all elements marked by $*$ vanishes, then the point configuration is said to be \emph{diagonal}.
\end{definition}

\begin{remark}
A configuration $A$ is an \emph{iterated circuit} in the sense of Esterov \cite[{Def.~3.15}]{Est10} 
if and only if after applying an integer affine transformation it is an upper diagonal
configuration all of whose diagonal configurations are circuits.
\end{remark}

\begin{theorem}
\label{thm:CharacteristicProduct}
Let $A$ be a diagonal point configuration with diagonal configurations $A_1$, \dots, $A_m$.
Then, $P_A(t) = P_{A_1}(t_1) \cdots P_{A_m}(t_m)$.
\end{theorem}

\begin{proof}
By induction, it suffices to consider the case $m=2$.
For a minor $|\hat A_\sigma|$ of $\hat A$ to be nonvanishing, it must hold that the set $\sigma \in \Sigma$
consists of $n_j$ columns corresponding to points in $\tilde A_j$ for $j=1,2$. Therefor, for each non-vanishing term
of $P_A(t)$ we have the factorization
\[
|A_\sigma| = |A_{\sigma_1}| \, |A_{\sigma_2}|
\]
where $\sigma_j \in \Sigma_j$ for $j=1,2$ and $\sigma = \sigma_1 \cup \sigma_2$.
We can write $A_j$ and its Gale dual in the block forms
\[
A_j = \left( \begin{array}{cc} 1 & 1 \\ 0 & \tilde A_j \end{array}\right)
\quad \text{ and } \quad
B_j = \left( \begin{array}{c} * \\ \tilde B_j \end{array}\right).
\]
Then, a Gale dual of $A$ can be written as
\[
B = \left(\begin{array}{cc}
 * & *  \\ 
 \tilde B_1 & 0 \\ 
 0 & \tilde B_2 
 \end{array}\right).
\]
Hence, for each $\sigma = \sigma_1 \cup \sigma_2$ it holds that
\[
\prod_{k\in \sigma} \< \beta_k, t\> = \left(\prod_{k\in \sigma_1} \< \beta_k, t_1\> \right)\left(\prod_{k\in \sigma_2} \< \beta_k, t_2\> \right),
\]
where we, by abuse of notation, interpret $\beta_k$ as a row both of $B$ and 
of $\tilde B_j$ when $k \in \sigma_j$.
\end{proof}

\begin{remark}
\label{rem:IteratedCircuitNonTrivial}
Let $A$ be an iterated circuit.
By Example~\ref{ex:PinCodimensionOne}, we have that 
$P_{A_j}(t_j)$ is a non-trivial monomial in the single variable $t_j$ of degree $n_j$ for each $j=1, \dots, m$.
It is straightforward to check that 
the coefficient of the monomial $t_1^{n_1}t_2^{n_2}\cdots t_m^{n_m}$ is unchanged if one deletes all the elements
of $A$ marked by a star in \eqref{eqn:DefIteratedCircuit}. In particular, it follows from
Theorem~\ref{thm:CharacteristicProduct}
that the monomial of $P_A(t)$ with exponent vector $(n_1, n_2, \dots, n_m)$ is given by
$P_{A_1}(t_1) \cdot P_{A_2}(t_2)\cdots P_{A_m}(t_m)$.
In particular, $P_A(t)$ is nontrivial in this case.
\end{remark}

\section{Linear Factors}
\label{sec:LinearFactors}

\begin{corollary}
\label{cor:LinearFormDividesP}
Let $\alpha \in A$.
If the cuspidal form $P_{A\setminus\{\alpha\}}$ is trivial, then $\<\beta_\alpha, t\>$ divides $P_A(t)$.
\end{corollary}

\begin{proof}
This follows from Theorem~\ref{thm:RetsrictionOfPToHyperplane} 
and Hilbert's Nullstellensatz.
\end{proof}

\begin{example}
Consider the point configuration and Gale dual
\[
A = \left( \begin{array}{ccccc}
1 & 1 & 1 & 1 & 1\\
0 & 1 & 2 & 0 & 1\\
0 & 0 & 0 & 1 & 2
\end{array}\right)
\quad \text{and} \quad
B^T = \left(\begin{array}{rrrrr}
2 & -1 & 0 & -2 & 1 \\ 1 & -2 & 1 & 0 & 0 
\end{array}\right).
\]
Deleting any of the two last columns of $A$ we obtain a pyramid,
which has trivial cuspidal form by Example~\ref{ex:PyramidHasTrivialCharacteristicPolynomial}.
Notice that the corresponding rows
of the Gale dual $B$ are parallel. The cuspidal form
in this case is $P_A(t) = 4 t_1(t_2-t_1)$. In particular, Corollary~\ref{cor:LinearFormDividesP} can not be extended to a bijective
correspondence between $\alpha \in A$ with $P_{A\setminus\{\alpha\}}$
trivial and linear factors of $P_A(t)$.
\end{example}

\begin{proposition}
\label{pro:ParallelRows}
Let $A$ be a point configuration 
such that $P_A(t)$ is non-trivial.
Assume that one (and therefor every) Gale dual $B$
has a set of\/ $k$ parallel rows $\beta_1, \dots, \beta_k$. Then,
$\<\beta, t\>^{k-1}$ divides $P_A(t)$.
If, in addition, $\beta_1 + \dots + \beta_k = 0$, then $\<\beta, t\>^k$ divides $P_A(t)$.
\end{proposition}

\begin{proof}
Let $\beta_j = \gamma_j \beta$ for some $\beta \neq 0$ and scalars $\gamma_j$ for $j= 1, \dots, k$.
Let the columns $\alpha_1, \dots, \alpha_k$ of $A$ correspond do the $k$ rows $\beta_1, \dots, \beta_k$
of $B$. By Lemma~\ref{lem:Schur}, any maximal minor of $A$ not containing
all but at most one of the columns $\alpha_1, \dots, \alpha_k$ vanishes.
In particular, the remaining points $\hat \alpha_{k+1}, \dots, \hat \alpha_N$ are
contained in some $n-k+1$-dimensional affine subspace. Since $A$ has full rank we can write
$A$ in the form (where we have reordered so that $\alpha_1, \dots, \alpha_k$ are the $k$ last columns)
\begin{equation}
\label{eqn:ParallelRowsOfBMatrices}
A = \left(\begin{array}{ccc}  1 & 1 & 1 \\   \hat A' & * & * \\   0 & I_{k-1} & \gamma \end{array}\right)
\quad \text{and} \quad
B = \left(\begin{array}{cc}  B' & * \\ 0 & - \gamma \\ 0 & 1\end{array}\right).
\end{equation}
Here, $A'$ is an $(n-k+2)\times (N-k)$-matrix with Gale dual $B'$, and $\gamma \in \RR^{k-1}$.
Set $|\gamma| = \gamma_1 + \dots + \gamma_{k-1}$.
We have that $\gamma_j \neq 0$ for each $j$, as $A$ is not a pyramid.
Under these assumptions $\beta = (0, \dots, 0, 1)$,
so that $\<\beta, t\> = t_m$.

To prove the first part of the proposition, we note that any maximal minor of $\hat A$ which does not contain at least $k-1$ of the rightmost $k$ columns vanishes. 
Hence, each nonvanishing term of \eqref{eqn:CharacteristicPolynomial} is divisible by $t_m^{k-1}$.

Let us now compute the coefficient of $t_m^{k-1}$. 
If we pick a maximal minor of $\hat A$ containing all of the $k$ last columns, then the corresponding monomial is divisible by $t_m^k$. Hence, we need only to consider the maximal minors of $\hat A$ containing exactly $k-1$ of the last $k$ columns.
The determinant of a $(k-1)\times(k-1)$ submatrix of $(I_{k-1}, \gamma)$
obtained by deleting the $j$th column is, up to sign, equal to $\gamma_j$ if $j = 1, \dots, {k-1}$, 
and it is equal to $1$ if $j = k$. 
Notice, also, that these $k-1$ columns gives a factor $t_m^{k-1}$ of the corresponding term of $P_A(t)$. Hence, when computing the coefficient of $t_m^{k-1}$ in $P_A(t)$ one should, for the remaining $n-k+1$ columns of the maximal minor in question, replace the factor $\<\beta_\alpha, t\>$
by $\<\beta_\alpha', t'\>$ where $t' = (t_1, \dots, t_{m-1})$ and $\beta'_\alpha$
is the corresponding row of $B'$.
Let $\Sigma'$ denote the set of all subsets
of $[N-k]$ of size $n-k+1$. 
All in all, we find that the coefficient of $t_m^{k-1}$ in the cuspidal form $P_A(t)$ is
\[
\sum_{\sigma \in \Sigma'} \, \left| \hat A'_\sigma \right|^2\,\Big( \prod_{j \in \sigma} \<\beta'_j, t'\>\Big)\,
\left( \bigg(\prod_{j=1}^{k-1} (-\gamma_j) \bigg) + \sum_{i=1}^{k-1} \gamma_i^2 \prod_{j \neq i} (-\gamma_j)  \right) = 
(-1)^{k-1}\,\gamma_1 \cdots \gamma_{k-1} \,\left( 1 - |\gamma|\right) \, P_{\hat A}(t').
\]
Thus, if $|\gamma| = 1$, which is equivalent to the original assumption that
 $\beta_1 + \dots + \beta_k = 0$,
then $P_A(t)$ is divisible by $t_m^k$.
\end{proof}

\begin{example}
Let us stay in the situation considered in Proposition~\ref{pro:ParallelRows}.
The coefficient of $t_m^{k-1}$ obtained at the end of that proof is
\[
(-1)^{k-1}\,\gamma_1 \cdots \gamma_{k-1} \,\left( 1 - |\gamma|\right) \, P_{\hat A}(t'),
\]
where $\gamma_j \neq 0$ for $j = 1, \dots, {k-1}$. It can happen that
this coefficient vanishes even if $|\gamma| \neq 1$, as  $P_{\hat A}(t')$
can be trivial. For example, consider the point configuration
and Gale dual
\[
A = \left( \begin{array}{ccccccccc}
1 & 1 & 1 & 1 & 1 & 1 & 1 & 1 & 1\\
0 & 1 & 0 & 0 & 1 & 1 & 0 & 0 & 1\\
0 & 0 & 1 & 0 & 1 & 0 & 0 & 0 & 1\\
0 & 0 & 0 & 1 & 0 & 1 & 0 & 0 & 1\\
0 & 0 & 0 & 0 & 0 & 0 & 1 & 0 & 3\\
0 & 0 & 0 & 0 & 0 & 0 & 0 & 1 & 2
\end{array}\right)
\quad \text{and} \quad
B = \left(\begin{array}{rrr}
1 & 1 & 7 \\
-1 & -1 & -1 \\
-1 & 0 & -1\\
0 & -1 & -1\\
1 & 0 & 0 \\
0 & 1 & 0 \\
0 & 0 & -3\\
0 & 0 & -2\\
0 & 0 & 1
\end{array}\right).
\]
Here, that last three rows of the Gale dual are parallel,
but they do not sum to zero.
The cuspidal form corresponding to this choice of Gale dual is
\[
P_A(t) = 6 \,t_3^3\, \left(7 t_3^2 + 2t_2 t_3 - 5 t_2^2 + 2t_1t_3 + 2t_1t_2-5t_1^2\right).
\]
\end{example}

\section{Defect Duals}
\label{sec:DualDefect}

A point configuration $A$ and, in the algebraic case, the toric variety $X_A$,
 is said to be \emph{dual defective} if the dual variety $\cX_A$
has codimension at least $2$. The following theorem is an immediate consequence
of the properties of the Jacobian matrix.

\begin{theorem}
\label{thm:DualDefectTrivialP}
The point configuration $A$ is dual defective if and only if\/ $P_A(t)$ is trivial.
\end{theorem}

\begin{proof}
If $A$ is dual defective, then the rank of the Jacobian matrix $J(w;t)$
is at most $n+m-1$ for generic $(w;t)$. It follows that its maximal minors,
which are polynomial in $(w;t)$, vanishes for generic $(w;t)$, and hence
they vanishes identically.
Conversely, if $\cX_A$ is a hypersurface, as the smooth locus of $\cX_A$
is nonempty, we can find a point $f = \Phi(w; t)$ 
for which the Jacobian matrix $J(w; t)$ has rank $n+m$.
\end{proof}

\begin{example}
Let $A$ be a pyramid. We saw in Example  \ref{ex:PyramidHasTrivialCharacteristicPolynomial}
that the polynomial $P_A(t)$ is trivial,
implying (the well-known result) that $X_A$ is dual defective. 
\end{example}

\begin{proposition}
\label{pro:SubsetNontrivial}
If $A$ has a subset $A'$, of full dimension, such that $A'$ is not dual defective,
then $A$ is not dual defective.
\end{proposition}

\begin{proof}
This (well-known result) follows from Corollary \ref{cor:ContainmentNontrivial}.
\end{proof}

Let us now recover Esterov's results on dual defective point configuration
from \cite[{Lem.~3.17}]{Est10} and \cite[{Cor.~3.20}]{Est13},
and extend it from the algebraic case to the case of exponential sums.

\begin{theorem}
\label{thm:DualDefectiveIteratedCircuit}
A specturm $A \subset \RR^{1+n}$ is dual defective if and only if $A$ does not contain any iterated circuit.
\end{theorem}

\begin{proof}
By Theorem~\ref{thm:DualDefectTrivialP} and Proposition~\ref{pro:SubsetNontrivial}
we translate to the equivalent statement that $P_A(t)$ is trivial
if and only if $A$ does not contain an iterated circuit. 
The \emph{only if}-direction was proven in Remark~\ref{rem:IteratedCircuitNonTrivial}.

Let us prove the \emph{if}-direction. Assume that $P_A(t)$ is nontrivial.
We use a double induction over the codimension $m$ and the dimension $n$,
where the base cases $m=1$ for arbitrary dimension is covered by Example~\ref{ex:PinCodimensionOne}.

If $P_{A\setminus \{\alpha\}}(t)$ is nontrivial for some $\alpha \in A$
then, by induction on codimension, $A \setminus \{\alpha\}$ contains an iterated circuit,
implying that $A$ contains an iterated circuit as well.
Hence, it suffices to consider the case when $P_{A\setminus \{\alpha\}}$ is trivial
for all $\alpha \in A$. 

By Corollary~\ref{cor:LinearFormDividesP}, 
we have that $\< \beta, t \>$ divides $P_A(t)$ for all rows $\beta$ of $B$.
However, $B$ has $N$ rows, while $P_A(t)$ is a homogeneous, nontrivial, polynomial of degree $n$. 
By the pigeon hole principle, the Gale dual $B$ must have parallel rows.
For a family $\beta_1, \dots, \beta_k$ of parallel rows of $B$,
where $k > 1$, let us borrow the notation from the proof of Proposition~\ref{pro:ParallelRows}:
that $\beta_j = \gamma_j \beta$ for some scalars 
$\gamma_j \neq 0$ for $j= 1, \dots, k$, and for some $\beta \neq 0$.
Then, again by a comparison of the number of rows of $B$ and the degree of $P_A(t)$,
there must exist a family, $\beta_1, \dots, \beta_k$ of parallel rows of $B$,
with $k > 1$, such that $\< \beta, t\>^k$ does not divide $P_A(t)$.
Let us fix such a family.

Let us write $A$ and $B$ as in \eqref{eqn:ParallelRowsOfBMatrices},
 and let $\gamma = (\gamma_1, \dots, \gamma_k)^T$ 
be as in that proof. Since $\<\beta, t\>^k$ does not divide $P_A(t)$,
we have that $|\gamma_k| \neq 1$.
Therefor, the $k\times k$-submatrix  of $A$ 
\[
\left(\begin{array}{cc} 1 & 1 \\ I_{k-1} & \gamma \end{array} \right)
\]
has rank $k$. (Recall that each $\gamma_k \neq 0$, for otherwise $A$ is a pyramid.) 
This has two consequences. First, we have that $(0, I_{k-1}, \gamma)$ is 
a $k$-dimensional circuit. 
Second, by applying an integer affine transformation,
we can eliminate all entries marked by $*$ in \eqref{eqn:ParallelRowsOfBMatrices}.
Thus, by Theorem~\ref{thm:CharacteristicProduct}, and with the notation of \eqref{eqn:ParallelRowsOfBMatrices}, it holds that $P_{A'}(t)$ is a factor
of $P_A(t)$.
In particular, $P_{A'}(t)$ is nontrivial, and hence 
$A'$ contains an iterated circuit of dimension $k-1$, by induction
on codimension and dimension. It follows that $A$ contains an iterated circuit.
\end{proof}

\section{Rationality of the Cuspidal Locus}

Let $A$ be algebraic.
Consider the case $n=1$, which was studied in detail in \cite{MT15},
and where the cuspidal form $P_A(t)$ appeared in this special case.
As $n=1$ the cuspidal form is, for every $m$, a non-trivial
linear form in $t$ vanishing along some hyperplane in $\PP^{m-1}$.
In particular, the cuspidal locus of $\cX_A$ is always unirational, and
if it is a subvariety of $\cX_A$ of codimension one, then it is rational.
Actually, a stronger statement holds: there is a point configuration $E$
such that the cuspidal locus of $\cX_A$ is isomorphic to the discriminantal variety $\cX_{E}$.
Let us here explain the corresponding result for general $n$.

\begin{theorem}
Let $A$ be algebraic.
For each linear factor $\<\ell, t\>$ of $P_A(t)$, such that $\ell$ is not proportional to any row
of the Gale dual $B$, there is an algebraic point configuration $E$ whose discriminant $\cX_E$
is birational to a subvariety of the cuspidal locus of\/ $\cX_A$.
\end{theorem}

\begin{proof}
In the algebraic case it suffices to consider the reduced discriminant
of Kapranov \cite{Kap91} mentioned in Remark~\ref{rem:ReducedDiscriminant}. 
Its Horn--Kapranov uniformazation $\Psi_A(t)$ is given by $t \mapsto \Phi(1; t)^B = (Bt)^B$.

Let $s$ denote coordinates
in $\ell^*$.
Let $\<\ell, t\>$ be a linear factor of $P_A(t)$,
where $\ell$ is not proportional to any row of the Gale dual $B$.
The hyperplane $\ell^*$ can be written in implicit form,
as the image a linear transformation $L\colon \PP^{m-2} \rightarrow \PP^{m-1}$.
Notice that $L$ can be chosen as an integer matrix.
Since $ABL = 0$, we can extend $A$ (by adding one extra row) to an integer matrix $E$ such that we 
have an exact sequence
\begin{center}
\begin{tikzpicture}
  \node (O0) {0};
  \node (O1) [right of=O0] {$\CC^{m-1}$};
  \node (O3) [right of=O1] {$\CC^{A}$};
  \node (O5) [right of=O3] {$\CC^{2+n}$};
  \node (O6) [right of=O5] {0.};
  \draw[->] (O0) to node {} (O1);
  \draw[->] (O1) to node {$BL$} (O3);
  \draw[->] (O3) to node {$E$} (O5);
  \draw[->] (O5) to node {} (O6);
\end{tikzpicture}
\end{center}
It follows that
$\left(\Psi_A \circ L\right)^L = \Psi_{E}$,
and hence $\left(\Psi_A \circ L\right)^L$ is a birational morphism from $\ell^*$ to $\cX_E$.
\end{proof}

The singular locus of $\cX_A$ is in general not birational to a discriminant
variety. That is, it is in general not rational. Let us give a simple example.

\begin{example}
Let $n=3$ and $N = 7$, so that $P_A(t)$ is a cubic form in three variables.
In particular, if $P_A(t)$ is non-singular then the cuspidal locus of $\cX_A$
is birational to a smooth cubic curve. Such a curve is
unirational but not rational. For an explicit example then $P_A(t)$
is non-singular we present the point configuration
\[
A = \left( \begin{array}{ccccccc}
1 & 1 & 1 & 1 & 1 & 1 & 1\\
0 & 1 & 1 & 2 & 3 & 3 & 3\\
3 & 0 & 2 & 2 & 1 & 2 & 3\\
0 & 3 & 2 & 2 & 2 & 3 & 0\end{array}\right).
\]
\end{example}

\section{The Bivariate Case: The Signature of the Quadratic Form}
\label{sec:Bivariate}

In the case that $n=2$ the polynomial
$P_A(t)$ is a quadratic form. Being defined only up to choice of coordinates,
let us in this section consider the (real) invariants given by the rank and the signature
of $P_A(t)$.

\begin{theorem}
\label{thm:QuadraticFormRank}
Let $n=2$ and let $s_A$ denote the
signature of\/ $P_A(t)$. Then, 
\begin{enumerate}[i)]
\item $s_A = (0,0; m)$ if and only if $A$ is a pyramid. \label{it:i}
\item $s_A = (1,0; 0)$ if and only if $A$ is contained in a non-real parabola. \label{it:ii}
\item $s_A = (0,1; m-1)$ if and only if $A$ is contained in a real parabola. \label{it:iii}
\item $s_A = (1,1; m-2)$ if and only if $A$ is contained in a hyperbola. \label{it:iv}
\item $s_A = (0,2; m-2)$ if and only if $A$ is contained in an ellipse. \label{it:v}
\item $s_A = (1,2; m-3)$ in all other cases. \label{it:vi}
\end{enumerate}
\end{theorem}

\begin{remark}
In the case of codimension $m=1$ it follows from Theorem~\ref{thm:QuadraticFormRank}
(and also from Example~\ref{ex:PinCodimensionOne}) that the combinatorial type of $A$
is determined by the signature of $P_A(t)$.
Indeed, $A$ is a \emph{simplex circuit} if and only if\/ $s_A = (1,0; 0)$,
it is a \emph{vertex circuit} if and only if\/ $s_A = (0,1; 0)$,
and it is a pyramid if and only if $s_A= (0, 0; 1)$.
\end{remark}

To prove Theorem~\ref{thm:QuadraticFormRank}, we write $A$ in the form
\begin{equation}
\label{eqn:QuadraticA}
A = \left( \begin{array} {cccccc}
1 & 1 & 1 & 1 & \dots & 1 \\
0 & 1 & 0 & \alpha_{11} & \cdots & \alpha_{m1} \\
0 & 0 & 1 & \alpha_{12} & \cdots & \alpha_{m2}
\end{array}\right).
\end{equation}
We write $\alpha_k = (\alpha_{k1}, \alpha_{k2})^T$, for simpler notation.
We set $|\alpha_k| = \alpha_{k1} + \alpha_{k2}$ for all $k$. We choose 
the dual matrix
\begin{equation}
\label{eqn:QuadraticB}
B^T = \left(\begin{array}{cccccccc}
|\alpha_1|-1 & -\alpha_{11} & -\alpha_{12} & 1 & 0 & \cdots & 0\\
|\alpha_2|-1 & -\alpha_{21} & -\alpha_{22} & 0 & 1 & \cdots & 0\\
\vdots & \vdots & \vdots & \vdots & \vdots & \ddots & \vdots \\
|\alpha_m|-1 & -\alpha_{m1} & -\alpha_{m2} & 0 & 0 & \cdots & 1
\end{array}\right).
\end{equation}

\begin{lemma}
\label{lem:BivariateQuadratixFormMatrix}
Let $n=2$, and let $A$ and $B$ be as in \eqref{eqn:QuadraticA} and \eqref{eqn:QuadraticB}.
Then, the quadratic form $P_A(t)$ is given by the matrix
$Q = \left( g(\alpha_k, \alpha_j) \right)_{k,j}$ where $1 \leq k,j \leq m$ and
\[
g(\alpha_k, \alpha_j) = \frac{1}{2} \bigg( \alpha_{k1} \alpha_{j2} (1-\alpha_{k1}-\alpha_{j2}) + \alpha_{k2} \alpha_{j1} (1-\alpha_{k2}-\alpha_{j1}) + (\alpha_{k1} \alpha_{j2} - \alpha_{k2} \alpha_{j1})^2\bigg).
\]
\end{lemma}

\begin{proof}
By Theorem~\ref{thm:RetsrictionOfPToHyperplane} it suffices to consider the case $m=2$,
which is a straightforward computation.
\end{proof}

Let us introduce notation for the following $k\times k$-minor of the matrix $Q$
from Lemma~\ref{lem:BivariateQuadratixFormMatrix}:
\[
G_k\,\Bigg(\begin{matrix} \alpha_1, \dots, \alpha_k \\ \delta_1, \dots, \delta_k \end{matrix}\Bigg)
 = \left| \begin{array}{ccc}
 g(\alpha_1, \delta_1) & \cdots &  g(\alpha_1, \delta_k)\\
 \vdots & \ddots & \vdots\\
  g(\alpha_k, \delta_1) & \cdots &  g(\alpha_k, \delta_k)
  \end{array}\right|.
\]
From this point on, proving the above statements is a matter of endurance during
computations. We avoid most details in this presentation.
We invite the reader to verify the following claims (preferrably using a cumputer).

\begin{lemma}
\label{lem:QuadraticG4}
The polynomial $G_k$ vanishes identically if\/ $k \geq 4$.
\end{lemma}

\begin{proof}
The case $k> 4$ follows from the case $k=4$ by a Laplace expansion.
The case $k=4$ is a straightforward computation.
\end{proof}

Thus, only the polynomials $G_2$ and $G_3$ are relevant for our investigation.
The polynomial $G_2$ has,
when expanded, $96$ terms. 
The polynomial $G_3$ is simpler; it admits a factorization
\begin{equation}
\label{eqn:G3}
G_3\,\Bigg(\begin{matrix} \alpha_1, \alpha_2, \alpha_3 \\ \delta_1, \delta_2, \delta_3 \end{matrix}\Bigg)
 = \frac 1 4 \,H(\alpha_1, \alpha_2, \alpha_3) \cdot H(\delta_1, \delta_2, \delta_3)
\end{equation}
where $H(\alpha_1, \alpha_2, \alpha_3)$ is the following polynomial with (when expanded) 24 terms:
\begin{align*}
H(\alpha_1, \alpha_2, \alpha_3)  &
= \alpha_{11} \alpha_{12} \alpha_{22}\alpha_{31}  (1 - \alpha_{22}) (1 - \alpha_{31})
 -  \alpha_{11} \alpha_{12} \alpha_{21} \alpha_{32} (1 - \alpha_{21})(1 - \alpha_{32})\\
&  +   \alpha_{12} \alpha_{21}\alpha_{31} \alpha_{32}(1 - \alpha_{12})  (1 - \alpha_{21} ) 
     -    \alpha_{12} \alpha_{21} \alpha_{22} \alpha_{31} (1 - \alpha_{12})(1 - \alpha_{31})\\   
&   -   \alpha_{11} \alpha_{22} \alpha_{31} \alpha_{32} (1 - \alpha_{11}) (1 - \alpha_{22})
   +    \alpha_{11}\alpha_{21} \alpha_{22} \alpha_{32}  (1 - \alpha_{11}) (1 - \alpha_{32}).
\end{align*}

\begin{lemma}
\label{lem:QuadraticParabola}
Assume that $0, e_1, e_2, \alpha_1$, and $\alpha_2$ are five distinct points in $\RR^2$.
Then, the polynomial $G_2$ vanishes for $\delta_1 = \alpha_1$ and $\delta_2 = \alpha_2$
if and only if there is a parabola containing the five points.
\end{lemma}

\begin{proof}
We assume that $\alpha_{11}(1-\alpha_{11}) \neq 0$, as this is the most difficult case.
The cases $\alpha_{11}=0$ or $\alpha_{11} = 1$ can be treated in analogous fashion.
The general equation of a parabola passing through the points $0, e_1$, and $e_2$ is
$P(x) = 0$ where 
\begin{equation}
\label{eqn:ParabolaLemma}
P(x) = a^2 x_1^2 + 2 a b x_1 x_2 + b^2 x_2^2 - a^2 x_1 -b^2x_2.
\end{equation}
Requiring in addition that the parabola passes through $\alpha_1$ gives that, in
projective coordinates
\begin{equation}
\label{eqn:FourPointsParabola}
[\, a \, : \, b \, ] = \big[\, \alpha_{11}\alpha_{12} \pm \sqrt{\alpha_{11}\alpha_{12}(\alpha_{11} + \alpha_{12}  -1)}\, : \, \alpha_{11} ( 1- \alpha_{11}) \big].
\end{equation}
(Recall that we assume $\alpha_{11} ( 1- \alpha_{11}) \neq 0$.)
Let $P_1(x)$ and $P_2(x)$ denote the two parabolic equations obtained 
from the possible choices of signs.
We leave it to the reader to verify that
\[
P_1(\alpha_2)P_2(\alpha_2) = 
4\, \alpha_{11}^2(1-\alpha_{11})^2\,  G_2\,\Bigg(\begin{matrix} \alpha_1, \alpha_2 \\ \alpha_1, \alpha_2 \end{matrix}\Bigg).\qedhere
\]
\end{proof}

\begin{remark}
\label{rem:G2}
Assume that $A$, written in the form \eqref{eqn:QuadraticA}, is contained in a 
parabola $C$.
For any indices $i,j,$ and $k$ we obtain a quadratic polynomial
\[
R_{ijk}(x) = G_2\Bigg(\begin{matrix} \alpha_i, \alpha_j \\ \alpha_k,\,\, x \end{matrix}\Bigg).
\]

Let us assume that $R_{121}$ is non-trivial, implying that the points $0, e_1, e_2, \alpha_1$,
and $\alpha_2$ are in general position (in the sense that there is a unique conic passing through them).
By Lemma~\ref{lem:QuadraticParabola}, $R_{121}(x)$ vanishes for $x = \alpha_2$. 
It is straightforward to verify that $R_{121}(x)$ also vanishes for 
$x = 0, e_1, e_2$, and $\alpha_1$. Hence, $R_{121}(x)$ defines the parabola $C$.
In particular, $R_{121}(x)$ vanishes for all $\alpha \in A$.

Assume now that $A$ has at least six points, and in addition that $R_{123}(x)$
is nontrivial.
We have that $R_{123}(x)$ vanishes at $x = \alpha_2$ and at $x = \alpha_3$,
as it coincides (up to a constant) with $P_{232}(\alpha_1)$ respectively $P_{323}(\alpha_1)$ for those values.
It is straightforward to check that $R_{123}(x)$ also vanishes for $x = 0, e_1,$ and $e_2$.
Hence, also $R_{123}(x)$ defines the parabola $C$.
In particular, $R_{123}(x)$ vanishes for all $\alpha \in A$.
\end{remark}

\begin{lemma}
\label{lem:QuadraticConic}
Assume that\/ $0, e_1, e_2, \alpha_1$, $\alpha_2$, and $\alpha_3$ are six distinct points in\/ $\RR^2$.
Then, the polynomial $H(\alpha_1, \alpha_2, \alpha_3)$ from \eqref{eqn:G3} vanishes
if and only if there is a conic containing all six points.
\end{lemma}

\begin{proof}
The general equation of a conic passing through the points $0, e_1$, and $e_2$ is
$P(x) = 0$ where 
\begin{equation}
\label{eqn:ParabolaLemma}
P(x) = a x_1^2 + b x_1 x_2 + c x_2^2 -a x_1 -cx_2
\end{equation}
Requiring in addition that the conic passes through $\alpha_1$ and $\alpha_2$
gives the coefficients, up to multiplication by a constant,
\[
\left\{\begin{array}{lll}
a & = &-\alpha_{12} \alpha_{22} (\alpha_{11} - \alpha_{21} + \alpha_{12} \alpha_{21} - \alpha_{11}\alpha_{22})\\
b & = &\alpha_{12} \alpha_{21}(1-\alpha_{12})(1-\alpha_{21}) 
 - \alpha_{11} \alpha_{22}(1-\alpha_{11})(1-\alpha_{22})\\
c & =  &\alpha_{11} \alpha_{21} (\alpha_{12} - \alpha_{12} \alpha_{21} - \alpha_{22} + \alpha_{11} \alpha_{22})
\end{array}\right.
\]
We leave it for the reader to verify that, with these coefficients, the polynomia $P(x)$
evaluated at $x = \alpha_{3}$ is equal to
$H(\alpha_1, \alpha_2, \alpha_3)$.
\end{proof}

\begin{proof}[Proof of Theorem~\ref{thm:QuadraticFormRank}]
It is well known that for $n=2$ the point configuration $A$ is dual defective if and only $A$ is a pyramid.
This proves part \eqref{it:i} of the theorem.

If the codimension $m=1$, then $A$ consists of four points in the plane, through which
two parabolas pass. It can be seen from \eqref{eqn:FourPointsParabola} that the limiting case
between real and non-real parabolas is the case of a pyramid.
Thus, this case follows from Example~\ref{ex:PinCodimensionOne}.

We now assume that $m > 1$. 
Classical geometry says that there is a unique conic passing through five points in generic
position (i.e., no four are colinear) in the plane. We allow the conic to be degenerate.
Assuming that $A$ is not a pyramid, it has a subconfiguration $A_1$
of five points in generic position.
Let $A_1$ constitute the first five columns of \eqref{eqn:QuadraticA}.

Let us first prove the relaxed statement, where we only consider the rank of
the cuspidal form $P_A(t)$. It follows from Lemma~\ref{lem:QuadraticG4} that the rank is at most three. 

If the rank is at most two, then by Lemma~\ref{lem:QuadraticConic} any choice of
six points of $A$ is contained in a conic. However, the five points of $A_1$
determine a unique conic $C$. By adjoining the remaining points one-by-one,
we conclude that $A$ is contained in the conic $C$. Conversely, if $A$ is contained in 
a conic, then $G_3$ vanishes for all $\alpha$ and $\delta$ by \eqref{eqn:G3}.

If the rank is at most one, then by Lemma~\ref{lem:QuadraticParabola} any choice of five points
is contained in a parabola. In particular, the five points of $A_1$ are contained in 
a parabola $C$. Let $\alpha$ be an additional point. Then, by Lemma~\ref{lem:QuadraticConic},
$A_1 \cup \{\alpha\}$ is contained in a conic. But $C$ is the unique conic containing $A_1$,
so $\alpha \in C$. It follows that $A$ is contained in $C$. Conversely, if $A$ is
contained in a parabola, then $G_2$ vanishes for all $\alpha$ and $\delta$
by Lemma~\ref{lem:QuadraticParabola} and Remark~\ref{rem:G2}.

Let us now turn to the refined statement of Theorem~\ref{thm:QuadraticFormRank}
regarding signatures.
For each class in the above list, it suffices to consider the minimal $m$ such that there is a
point configuration in this class. Indeed, there is a ``minimal'' subconfiguration
witnessing the class containing $A$, and by Theorem~\ref{thm:RetsrictionOfPToHyperplane}
we can delete the remaining points without altering the rank.
In particular, parts \eqref{it:ii} and \eqref{it:iii} follows from the above discussion
on codimension $m=1$.

Any two minimal configurations of one class can be continuously deformed to each other
without leaving the class in question. Therefor, it suffices to consider one representative of each class.
Since a generic configuration can be continuously deformed to a configuration contained
either in an ellipse or in a parabola, part \eqref{it:vi} follows from parts \eqref{it:iv}  and \eqref{it:v}
and the fact that the rank is $3$ in the generic case.
Hence, we finish the proof with Examples \ref{ex:Hyperbola} and  \ref{ex:Parabola}.
\end{proof}

\begin{example}
\label{ex:Hyperbola}
Consider the point configuration and Gale dual
\[
A = \left( \begin{array}{rrrrr}
1 & 1 & 1 & 1 & 1\\
3 & -3 & 5 & 5 & -5\\
0 & 0 & 4 & -4 & 4
\end{array}\right)
\quad \text{and} \quad
B^T = \left(\begin{array}{rrrrr}
5 & -5 & -3 & 0 & 3 \\ -8 & 2 & 3 & 3 & 0 
\end{array}\right).
\]
We obtain the cuspidal form
$P_A(t) = -576(5t_1- 4t_2)(5t_1 +4 t_2)$,
which has signature $s_A = (1,1; 0)$.
\end{example}

\begin{example}
\label{ex:Parabola}
Consider the point configuration and Gale dual
\[
A = \left( \begin{array}{ccccc}
1 & 1 & 1 & 1 & 1\\
0 & 1 & 0 & 1 & 2\\
0 & 0 & 1 & 2 & 1
\end{array}\right)
\quad \text{and} \quad
B^T = \left(\begin{array}{rrrrr}
2 & -2 & -1 & 0 & 1 \\ 2 & -1 & -2 & 1 & 0 
\end{array}\right).
\]
We obtain the cuspidal form
$P_A(t) = -4(t_1^2 + t_1t_2 + t_2^2)$,
which has signature $s_A = (0,2; 0)$.
\end{example}


\bibliographystyle{amsplain}

\renewcommand{\bibname}{References} 


\begin{thebibliography}{10}

\bibitem{BD16}
{Bihan, F.} and {Dickenstein, A.}, \emph{{Descartes' Rule of Signs for
  Polynomial Systems supported on Circuits}}, {arXiv:1601.05826}, {2016}.

\bibitem{DFS07}
{Dickenstein, A.}, {Feichtner, E. M.}, and {Sturmfels, B.}, \emph{{Tropical
  discriminants}}, {J. Amer. Math. Soc.} \textbf{{20}} ({2007}), no.~{4},
  {1111--1133}.

\bibitem{DP17}
{Dickenstein, A.} and {Piene, R.}, \emph{{Higher order selfdual toric
  varieties}}, {Ann. Mat. Pura Appl. (4)} \textbf{{196}} ({2017}), no.~{5},
  {1759--1777}.

\bibitem{Est10}
{Esterov, A.}, \emph{{Newton polyhedra of discriminants of projections}},
  {Dicrete Comput. Geom.} \textbf{{44}} ({2010}), no.~{1}, {96--148}.

\bibitem{Est13}
\bysame, \emph{{Characteristic classes of affine varieties and Pl{\"u}cker
  formulas for affine morphisms}}, {arXiv:1305.3234v7}, {2013}.

\bibitem{FNR17}
{Forsg{\aa}rd, J.}, {Nisse, M.}, and {Rojas, J. M.}, \emph{{New subexponential
  fewnomial hypersurface bounds}}, {arXiv:1710.00481}, {2017}.

\bibitem{GKZ94}
{Gel'fand, I. M.}, {Kapranov, M. M.}, and {Zelevinsky, A. V.},
  \emph{{Discriminants, resultants, and multidimensional determinants}},
  {Mathematics: Theory \& Applications}, {Birkh{\"a}user Boston, Inc.},
  {Boston, MA}, {1994}.

\bibitem{Kap91}
{Kapranov, M. M.}, \emph{{A characterization of $A$-discriminantal
  hypersurfaces in terms of the logarithmic Gauss map}}, {Math. Ann.}
  \textbf{{290}} ({1991}), no.~{2}, {277--285}.

\bibitem{Kat73a}
{Katz, N.}, \emph{{Pinceaux de Lefschetz: th{\'e}oreme d'{\'e}xistence}},
  {Lecture Notes in Mathematics}, vol. {340}, ch.~{XVII}, pp.~{212--253},
  {Springer-Verlag, Berlin}, {1973}.

\bibitem{Kho91}
{Khovanski{\u\i}, A. G.}, \emph{{Fewnomials}}, {Translations of Mathematical
  Monographs}, vol.~{88}, {American Mathematical Society, Providence, RI},
  {1991}, {Translated from the Russian by Smilka Zdravkovska}.

\bibitem{MT15}
{Mikhalkin, E. N.} and {Tsikh, A.}, \emph{{Singular strata of cuspidal type for
  the classical discriminant}}, {Mat. Sb.} \textbf{{206}} ({2015}), no.~{2},
  {119--148}, {(in Russian), translation in Sb. Math. \textbf{206} (2015), no.
  1--2, 282--310.}

\bibitem{RR16}
{Rojas, J. M.} and {Rusek, K.}, \emph{{$A$-Discriminants for Complex Exponents
  and Counting Real Isotopy Types}}, {arXiv:1612.03458}, {2016}.

\end{thebibliography}

\end{document}